\documentclass[11pt]{article}
\usepackage{tikz}
\usetikzlibrary{arrows.meta,arrows,decorations.pathreplacing}
\usepackage[linktocpage]{hyperref}

\usepackage{amsmath, amsfonts, amsthm, verbatim,
amssymb, 
dsfont,color}

\usepackage[utf8x]{inputenc}
\usepackage{lmodern}
\usepackage{mathtools}

\newcommand{\dd}{\mathrm{d}}
\newcommand{\R}{\mathbb{R}}

\newcommand{\Z}{\mathbb{Z}}
\newcommand{\E}{\mathbf{E}}
\newcommand{\p}{\mathbf{P}}
\renewcommand{\P}{\mathbf{P}}

\newcommand{\bone}{\mathds 1}

\newcommand{\trint}{\int \hspace*{-5pt} \int \hspace*{-5pt} \int}

\theoremstyle{plain}
\newtheorem{lemma}{Lemma}
\newtheorem{theorem}{Theorem}
\newtheorem{proposition}{Proposition}
\newtheorem{corollary}{Corollary}

\theoremstyle{remark}
\newtheorem{example}{Example}
\newtheorem{remark}{Remark}

\newcommand{\calb}{{\cal B}}

\newcommand{\al}{{\alpha}}
\newcommand{\la}{{\lambda}}
\newcommand{\La}{{\Lambda}}
\newcommand{\eps}{{\varepsilon}}

\newcommand{\ov}{\overline}

\newcommand{\ii}{\mathrm{i}}

\newcommand{\bthm}{\begin{theorem}}
\newcommand{\ethm}{\end{theorem}}

\newcommand{\bcor}{\begin{corollary}}
\newcommand{\ecor}{\end{corollary}}

\newcommand{\blem}{\begin{lemma}}
\newcommand{\elem}{\end{lemma}}

\newcommand{\bprop}{\begin{proposition}}
\newcommand{\eprop}{\end{proposition}}

\newcommand{\bdf}{\begin{definition}}
\newcommand{\edf}{\end{definition}}

\newcommand{\bex}{\begin{example}}
\newcommand{\eex}{\end{example}}

\newcommand{\brem}{\begin{remark}}
\newcommand{\erem}{\end{remark}}

\newcommand{\bpr}{\begin{proof}}
\newcommand{\epr}{\end{proof}}

\newcommand{\benu}{\begin{enumerate}}
\newcommand{\eenu}{\end{enumerate}}

\newcommand{\beq}{\begin{equation}}
\newcommand{\eeq}{\end{equation}}

\newcommand{\bit}{\begin{itemize}}
\newcommand{\eit}{\end{itemize}}

\numberwithin{equation}{section}

\title{Extremes of the stochastic 
heat equation with additive Lévy noise}

\author{
Carsten Chong\thanks{
Department of Statistics, Columbia University, 1255 Amsterdam Avenue, 
New York, NY 10027, USA, 
e-mail: carsten.chong@columbia.edu} \and P\'eter 
Kevei\thanks{{Bolyai Institute, University of Szeged,
Aradi v\'ertan\'uk tere 1, 6720 Szeged, Hungary, 
e-mail: kevei@math.u-szeged.hu}}
}

\begin{document}

\date{}

\maketitle

\begin{abstract}
We analyze the spatial asymptotic properties of the solution to the stochastic
heat equation driven by an additive L\'evy space-time white noise. 
For fixed time $t > 0$ and space $x \in \R^d$ we determine the exact
tail behavior of the solution both for light-tailed and for heavy-tailed 
L\'evy jump measures. Based on these asymptotics we determine for any 
fixed time $t> 0$ the almost-sure growth rate of the solution as 
$|x| \to \infty$. \\

 \emph{MSC2020 subject classifications:}  Primary:   60H15; 60F15; 60G70; secondary: 60G17, 60G51.
 
\end{abstract}

\section{Introduction} \label{sect:intro}

We consider the stochastic heat equation (SHE) on $\R^d$ driven 
by an additive L\'evy space-time  white noise  $\dot \Lambda$, 
with zero initial condition, given by
\beq\label{eq:SHE}
\begin{split}
\partial_t Y(t,x) &= \frac{\kappa}{2} \Delta Y(t,x) + 
\dot\Lambda (t,x),\qquad 
(t,x)\in(0,\infty)\times \R^d,\\
Y(0, \cdot)&= 0,
\end{split}
\eeq
where $\Delta$ stands for the Laplacian, 
$\kappa > 0$ is the diffusion constant,
and the measure $\Lambda$ is given by
\begin{equation} \label{eq:def-Lambda}
\Lambda(\dd t, \dd x)  =  m\, \dd t \dd x + 
\int_{(1,\infty)} z \, \mu(\dd t, \dd x, \dd z) 
+ \int_{(0,1]} z \, ( \mu - \nu) (\dd t, \dd x, \dd z).
\end{equation}
Here, $m \in \R$ and $\mu$ is a Poisson random measure on 
$(0, \infty) \times \R^d \times (0,\infty)$ whose intensity measure $\nu$ takes the form
$\nu ( \dd t, \dd x, \dd z) = \dd t\, \dd x\, \lambda(\dd z)$,
with a L\'evy measure   satisfying $\int_{(0,\infty)} (1\wedge z^2)\,\la(\dd z)<\infty$.
To exclude trivialities, we always assume that $\lambda$ is not identically zero.

In this  case the mild solution to \eqref{eq:SHE} can be written explicitly 
in the form
\begin{equation} \label{eq:defY}
Y(t,x) = 
\int_0^t \int_{\R^d} g(t-s, x-y)\,  \Lambda(\dd s, \dd y),
\end{equation}
where
\begin{equation} \label{eq:def-heat}
g(t,x) = \frac{1}{(2 \pi \kappa t)^{d/2}} e^{-\frac{|x|^2}{2 \kappa t}}, \qquad
t > 0,~ x \in \R^d,
\end{equation}
is the heat kernel. 
In our earlier paper \cite{CK2} we obtained a complete description of the  
almost-sure growth behavior of $Y(t,x)$ at a fixed spatial point $x \in \R^d$ as $t \to \infty$. In particular, $t\mapsto Y(t,x)$ satisfies a weak law of large numbers but surprisingly violates the strong law of large numbers. In the present paper we continue these investigations and 
analyze the almost-sure behavior for fixed time $t > 0$, as $|x| \to \infty$. 

\smallskip
To this end,  we determine  in Section~\ref{sect:tail}
the exact tail asymptotics for $Y(t,x)$ 
both for light-tailed and for heavy-tailed L\'evy measures. Note that since the heat kernel is singular at the origin, the results of \cite{Fasen05,Fasen06,Rootzen78} for moving-average processes with bounded kernels do not apply.
In \cite{CK1} we proved that for any jump measure $\lambda$, the $(1+\frac 2d)$-moment of $Y(t,x)$ is infinite, which suggests a power-law tail behavior.
In Theorem \ref{thm:Ytail} we show that this is indeed the case, regardless of whether the noise itself is light- or heavy-tailed.
Section \ref{sect:spatial} contains the tail asymptotics for 
$\sup_{x \in A} Y(t,x)$, where $A$ is a bounded Borel set. Based on these results, 
we determine in Section \ref{sect:growth} the almost-sure growth behavior of $Y(t,x)$ as $|x| \to \infty$.
The behavior is very different from the behavior of the Gaussian case, in which
\beq\label{eq:Gauss}
\limsup_{\lvert x \rvert\to\infty} \dfrac{Y(t,x)}{(\log |x|)^{1/2}} = 
\biggl ( \frac{4t}{\pi\kappa} \biggr)^{\frac14} 
\eeq
almost surely; see \cite[Eq.~(6.3)]{Khoshnevisan17}. 
All the proofs are gathered together in Section \ref{sect:proofs}.
In our companion paper \cite{islands_mult} we address the same questions for the SHE with multiplicative Lévy noise.

Let us end this introductory section by stating
 necessary and sufficient conditions  in terms of the jump 
measure $\lambda$ for the existence of the integral \eqref{eq:defY}.
To our best knowledge, this result is new. While many works have  studied sufficient conditions for  existence \cite{Balan14, Chong, Chong1, SLB98}, necessary and sufficient conditions have only been derived for multiplicative noise \cite{Berger21b} or for specific types of noises such as $\al$-stable noise \cite{Dalang19}. 
Introduce the measure $\eta$ as 
\begin{equation} \label{eq:def-eta}
\eta (B) = \nu \Bigl( \{ (s,y,z) : \, s \leq t, \, g(s,y) z  \in B \} \Bigr),
\end{equation}
where $B \subseteq (0,\infty)$ is a Borel set. 

\begin{theorem} \label{thm:existence} Suppose that $\La$ is of the form \eqref{eq:def-Lambda}.

(i)
The integral defining $Y(t,x)$ in \eqref{eq:defY} exists if and only if (iff)
\begin{equation} \label{eq:existence-cond}
\int_{(1,\infty)} (\log z)^{d/2} \,\lambda( \dd z) < \infty
\qquad \text{and } \qquad
\begin{cases}
\int_{(0,1]} z^2 \,\lambda(\dd z) < \infty & \text{if }d = 1, \\
\int_{(0,1]} z^2 \lvert \log z\rvert\, \lambda(\dd z) < \infty &\text{if } d = 2, \\
\int_{(0,1]} z^{1+2/d}\, \lambda(\dd z) < \infty&\text{if } d \geq 3. \\
\end{cases}
\end{equation}
In this case, $\eta$ is a Lévy measure and $Y(t,x)$  is infinitely divisible with 
characteristic function
\begin{equation} \label{eq:charfunc}
\E [e^{\ii \theta Y(t,x) }] = 
\exp \left\{ \ii \theta A +
\int_{(0,\infty)} 
\left( e^{\ii \theta u} - 1 - \ii \theta u \bone(u \leq 1) \right) \eta ( \dd u)
\right\},
\end{equation}
where $\bone$ stands for the indicator function and
$A \in \R$ is some explicit constant.

(ii) The integral 
\begin{equation} \label{eq:uncomp}
\int_0^t \int_{\R^d} \int_{(0,\infty)} g(t-s,x-y) z\, \mu(\dd s, \dd y, \dd z)
\end{equation}
exists iff
\begin{equation} \label{eq:existence-cond-uncomp}
\int_{(1,\infty)} (\log z)^{d/2} \,\lambda( \dd z) < \infty \qquad \text{and} \qquad
\int_{(0,1]} z\, \lambda(\dd z) < \infty.
\end{equation}
\end{theorem}

\brem Note that \eqref{eq:existence-cond} is identical to the necessary and sufficient condition found in \cite{Berger21b} for the existence of solutions to the SHE with multiplicative noise in dimensions $d=1,2$ but is weaker than the necessary condition found in \cite[Prop.\ 2.2]{Berger21b} for $d\geq3$. In other words, if $d\geq3$, there are Lévy noises for which the SHE with additive noise has a solution but the SHE with multiplicative noise does not.
\erem

Whenever $\int_{(0,1]} z \,\lambda(\dd z) < \infty$,  there is no need 
for compensation, so we assume without loss of generality that $m = \int_{(0,1]} z\, \lambda(\dd z)$. In this case,  
\begin{equation} \label{eq:defY3}
\begin{split}
Y(t,x)  = \int_0^t \int_{\R^d} \int_{(0,\infty)} g(t-s, x-y) z \, 
\mu (\dd s, \dd y, \dd z)  = \sum_{\tau_i \leq t} g(t-\tau_i, x - \eta_i) \zeta_i,
\end{split}
\end{equation}
where $(\tau_i, \eta_i, \zeta_i)$ are the points of the Poisson random measure
$\mu$.
In what follows we always assume that \eqref{eq:existence-cond} holds.

\section{Tail asymptotics} \label{sect:tail}

Since $Y(t,x)$ is infinitely divisible, its tail behavior
is the same as the tail behavior of its L\'evy measure $\eta$, whenever 
the tail is subexponential. This result was proved by  \cite{EGV}
for nonnegative infinitely divisible random variables and by 
\cite{Pakes, Pakes2} in the general case. Therefore, we need to
determine the tail of the L\'evy measure $\eta$ in \eqref{eq:def-eta}.

For $\gamma > 0$
introduce the moments and truncated moments of $\lambda$ as
\begin{equation} \label{eq:M-def}
m_\gamma(\lambda) = \int_{(0,\infty)} z^\gamma\, \lambda(\dd z) \qquad \text{and}\qquad
M_\gamma(x) = \int_{(0,x]} z^\gamma \,\lambda(\dd z).
\end{equation}

\begin{lemma} \label{lemma:eta-form}
Let $D = ( 2 \pi \kappa t)^{d/2}$. For any $r > 0$,
\begin{equation} \label{eq:eta-tail}
\overline \eta(r)  = \eta ((r,\infty) ) = 
r^{-(1 + 2/d)}
\frac{ d^{d/2}}{\pi \kappa (d + 2)^{d/2 +1} \Gamma(\frac d2+1)} 
\int_0^\infty e^{-u} u^{d/2} M_{1+2/d}(Dr e^{ud/(d+2)}) \,\dd u.
\end{equation} 
\end{lemma}

From the representation above we immediately see that 
as soon as $m_{1+2/d}(\lambda) < \infty$, then $\overline \eta(r)
\sim c \,r^{-1-2/d}$, for some $c > 0$. We can determine the tail
even if this moment condition does not hold, provided that $\overline \lambda$
is regularly varying.

In the following, the class of regularly varying functions with index $\rho \in \R$
is denoted by $\mathcal{RV}_\rho$. For general theory on regular
variation we refer to  \cite{BGT}.
Write $\overline \lambda(r) = \lambda((r,\infty))$.
By Karamata's theorem, for $\al>0$,
$\overline \lambda \in \mathcal{RV}_{-\alpha}$ 
iff the truncated moment $M_{1+2/d}$ in \eqref{eq:M-def} 
is also regularly varying. However, for $\alpha = 0$,
the latter holds iff $\overline \lambda $ belongs to the 
de Haan class (see e.g.~\cite[Thm.\ 3.7.1]{BGT}).
Therefore, it is more difficult to determine the 
asymptotics of $\overline \eta$ for $\alpha = 0$, and in fact the 
result itself is surprising.

\begin{lemma} \label{lemma:eta-tail} Let $\la$ satisfy \eqref{eq:existence-cond}.

(i) Assume that $m_{1+2/d}(\lambda) < \infty$. Then 
\[
\overline \eta(r) \sim 
r^{-1 -2/d}
\frac{ d^{d/2}}{\pi \kappa (d + 2)^{d/2 +1} } m_{1+2/d}(\lambda),
\qquad r \to \infty.
\]

(ii) Assume that $\overline \lambda (r) = \ell(r) r^{-\alpha}$ 
for $\alpha \in (0,1+\frac 2d]$, where $\ell$ is slowly varying,
and if $\alpha = 1 + \frac 2d$, further assume that 
$\int_1^\infty \ell(u) u^{-1} \,\dd u = \infty$.
Then as $r \to \infty$,
\[
\overline \eta (r) \sim 
\begin{dcases}
\ell(r) r^{-\alpha}
\frac{D^{1+2/d-\alpha}}{d \pi \kappa \alpha^{d/2} (1+\tfrac 2d - \alpha)}
&\text{if } \alpha < 1 +\tfrac 2d, \\
L(r) r^{-1-2/d}
(d \pi \kappa (1+\tfrac2d)^{d/2})^{-1}
& \text{if } \alpha = 1 + \tfrac 2d,
\end{dcases}
\]
where 
\begin{equation} \label{eq:def-L}
L(r) = \int_1^r \ell(u) u^{-1}\, \dd u.
\end{equation}

(iii)
Assume that $\al=0$ and $\overline \lambda(x) = \ell(x)$ is slowly varying. Then as $r \to \infty$,
\[
\overline \eta(r) \sim  L_0(r) 
\frac{D^{1+2/d}}{2 \pi \kappa \Gamma(\frac d2+1) ( 1 + \frac 2d)},
\]
where 
\[
L_0(r) 
:= \int_1^\infty \ell(ry) y^{-1} (\log y)^{d/2-1}  \,\dd y
\]
is slowly varying and 
$L_0(r) / \ell(r) \to \infty$ as $r \to \infty$.
\end{lemma}

\begin{example}
Assume that  $\overline \lambda(r) = (\log r)^{-\beta}$ for $r > e$. Then
 \eqref{eq:existence-cond} holds iff $\beta > \frac d2$.
By substituting $u=(1+\log y/\log r)^{-1}$, we obtain
\[
L_0(r) = (\log r)^{d/2 - \beta} \, \mathrm{B}(\tfrac d2, \beta - \tfrac d2),\qquad r\to\infty,
\]
where $\mathrm{B}$ is the usual beta function.
\end{example}

To determine the tail of the spatial supremum, we need the tail
of the largest contribution to $Y(t,x)$ by a single atom. Without loss of generality, consider $x=0$ and
define
\begin{equation} \label{eq:def-barY}
\overline Y(t) = \sup_{\tau_i \leq t} g(t-\tau_i, \eta_i) \zeta_i. 
\end{equation}
For $r > 0$ large, let 
\[
S_r = \{ (s, y, z) : s \in [0,t], g(s, y) z > r \}. 
\]
Clearly, $\overline Y(t) \leq r$ iff $\mu(S_r) = 0$, which shows that
\begin{equation} \label{eq:tail-barY}
\p ( \overline Y(t) \leq r ) = e^{-\nu(S_r)} = e^{-\overline \eta(r)}.
\end{equation}
As a result we obtain the following.

\begin{theorem} \label{thm:Ytail}
Let $Y(t,x)$ be given in \eqref{eq:defY} and assume \eqref{eq:existence-cond}. 

(i) The tail function $\ov\eta$ has extended regular variation at infinity \cite[p.\ 65]{BGT}, that is, there are $\theta_1,\theta_2\in\R$ such that for any $c>1$, 
\beq\label{eq:O} c^{\theta_1}\leq \liminf_{x\to\infty} \frac{\ov\eta(cx)}{\ov\eta(x)} \leq\limsup_{x\to\infty} \frac{\ov\eta(cx)}{\ov\eta(x)} \leq c^{\theta_2}.\eeq 

(ii) As $r \to \infty$,
\beq\label{eq:barY}
\p ( Y(t,x) > r) \sim  \p ( \overline Y (t)> r) \sim \overline \eta(r).
\eeq

(iii) For $\alpha \in [0,1+\frac2d)$, $\overline \eta \in \mathcal{RV}_{-\alpha}$
iff $\overline \lambda \in \mathcal{RV}_{-\alpha}$. For $\al=1+\frac2d$, we have $\ov\eta\in\mathcal{RV}_{-1-2/d}$ iff $r\mapsto \int_0^r u^{2/d} \int_{1}^\infty  {\overline \lambda(uv)} 
 ( \log v  )^{d/2-1}v^{-1}\, \dd v\, \dd u$ is slowly varying. In particular, the latter holds if $m_{1+2/d}(\la)<\infty$.
\end{theorem}

\section{Spatial supremum} \label{sect:spatial}

Let $A \in \mathcal{B}(\R^d)$ be a Borel subset of $\R^d$
with finite and positive Lebesgue measure
and define
\begin{equation} \label{eq:XA}
X_A (t) = \begin{dcases} \sum_{\eta_i \in \ov A, \tau_i \leq t} (2\pi\kappa 
(t-\tau_i))^{-d/2} \zeta_i\bone_{\{(2\pi\kappa 
(t-\tau_i))^{-d/2} \zeta_i>1\}} &\text{if } d=1,\\ 
\sum_{\eta_i \in \ov A, \tau_i \leq t} 
(2\pi\kappa 
(t-\tau_i))^{-d/2} \zeta_i &\text{if } d\geq2, \end{dcases}
\end{equation}
where $\ov A$ is the closure of $A$.
Since $X_A(t)$ is a functional of a Poisson random measure, one  easily obtains
necessary and sufficient conditions for the existence.

Define the measure $\tau$ as 
\begin{equation} \label{eq:deftau}
\tau(B) =  ( \mathrm{Leb} \times \lambda  ) \Bigl( 
\{ (s,z) : \, (2\pi\kappa s)^{-d/2} z \in B\cap(\bone_{\{d=1\}},\infty),\ s \leq t \}
\Bigr),\qquad B\in\calb(\R^d),
\end{equation}
where  $\mathrm{Leb}$ is the Lebesgue measure. For a Borel set $A$ let 
$\lvert A\rvert$ be its Lebesgue measure.

\begin{theorem} \label{thm:existence-X} Suppose that $\lvert \ov A\rvert\in(0,\infty)$.
The sum $X_A(t)$   is finite a.s.\ iff
\begin{equation} \label{eq:lambda2d}
\int_{(0,1)} z^{2/d}\lvert\log z\rvert^{\bone_{\{d=2\}}} \,\lambda(\dd z ) < 
\infty.
\end{equation}
Furthermore, if \eqref{eq:lambda2d} holds then 
\begin{equation} \label{eq:chf-X}
\E [e^{\ii \theta X_A(t) }] = 
\exp \left\{ |\ov A|  
\int_{(0,\infty)}  ( 1 - e^{-\ii \theta u}  )\, \tau (\dd u)
\right\}.
\end{equation}
\end{theorem}

Note that \eqref{eq:lambda2d} holds for any Lévy measure if $d=1$. 
From \eqref{eq:deftau} we obtain that for $r > 1$
\begin{equation} \label{eq:tau-form}
\begin{split}
\overline \tau (r)  = \tau ( (r,\infty)) &=
\int_{(0,\infty)}  \Bigl((2\pi\kappa)^{-1}( z/r  )^{2/d} \wedge t\Bigr) \, \lambda(\dd z )\\
&  =  r^{-2/d} (2\pi\kappa)^{-1}M_{2/d}(r D) 
+  t\, \overline \lambda(r D)  \\
& = \frac{1}{\pi \kappa d} r^{-2/d} 
\int_0^{rD} u^{2/d -1} \overline \lambda(u) \,\dd u.
\end{split}
\end{equation}

In specific cases, we can determine the asymptotic behavior of $\ov\tau$ explicitly.

\begin{lemma} \label{lemma:tau-tail}
Assume \eqref{eq:lambda2d}.

(i) If $m_{2/d}(\lambda) < \infty$, then 
$\overline \tau (r) \sim (2\pi\kappa)^{-1} m_{2/d}(\lambda) r^{-2/d}$ as $r \to \infty$.

(ii) Assume that $\overline \lambda (r) = \ell(r) r^{-\alpha}$ 
for $\alpha \in [0,\frac 2d]$, where $\ell$ is slowly varying,
and further assume $\int_1^\infty \ell(u) u^{-1} \,\dd u = \infty$ if $\alpha = \frac 2d$.
Recalling the definition of $L$ from \eqref{eq:def-L}, we have as $r \to \infty$ that
\[
\overline \tau (r) \sim 
\begin{dcases}
\frac{ 2 t D^{-\al}}{2 - d \alpha} \ell(r) r^{-\alpha}
&\text{if } \alpha < \tfrac 2d, \\
\tfrac{2}{d}(2\pi\kappa)^{-1} L(r) r^{-2/d} & \text{if } \alpha = \tfrac 2d.
\end{dcases}
\]
\end{lemma}

Introduce the notation
\begin{equation} \label{eq:barX}
\overline X_A(t) = \sup \left\{ 
(2\pi\kappa(t-\tau_i))^{-d/2} \zeta_i \, :
\tau_i \leq t, \eta_i \in \ov A \right\}.
\end{equation}
To determine the tail of $\overline X_A(t)$, let  
$
T_r =  \{ (s,z) : s \leq t, \ (2\pi\kappa s)^{-d/2} z > r  \}.
$
Then $\overline X_A (t) \leq r$ iff $\mu( A \times T_r) = 0$, thus
\begin{equation} \label{eq:tail-barX}
\p ( \overline X_A(t) > r ) = 1 -  e^{- \lvert \ov A\rvert \overline \tau(r)}. 
\end{equation}

\begin{theorem} \label{thm:Xtail}
Assume \eqref{eq:lambda2d}.

(i) The tail function $\ov\tau$ has extended regular variation at infinity.

(ii) For every bounded Borel set $A$,
\beq\label{eq:XA2}
\p ( X_A(t) > r) \sim \p (\overline X_A(t) > r) \sim
\lvert \ov A\rvert \, \overline \tau(r),\qquad r\to\infty.
\eeq

(iii) For $\alpha \in [0,\frac2d)$, $\overline \tau \in \mathcal{RV}_{-\alpha}$
iff $\overline \lambda \in \mathcal{RV}_{-\alpha}$.
For $\al=\frac2d$, we have $\overline \tau \in\mathcal{RV}_{-2/d}$ iff 
$r\mapsto \int_0^r u^{2/d-1} \overline \lambda(u) \, \dd u$ is slowly varying. In 
particular, the latter holds if $m_{2/d}(\la)<\infty$ or if $\ov\la\in\mathcal{RV}_{-2/d}$.
\end{theorem}

In order to  determine the tail asymptotics of the local supremum of the solution, let us introduce for each $A\in\calb(\R^d)$ the measure
\begin{equation} \label{eq:def-eta-A}
\eta_A (B) = \nu \Bigl( \{ (s,y,z) : \, s \leq t, \, (2\pi s)^{-d/2}e^{-\frac{\mathrm{dist}(y,A)^2}{2\kappa s}} z  \in B \cap(\bone_{\{d=1\}},\infty)\} \Bigr),\qquad B\in\calb(\R^d).
\end{equation} 

If $m_{2/d}(\la)<\infty$ or if $\ov\la$ is regularly varying with index $-\al$ for some $\al\in(0,\frac2d]$, one can express $\ov\eta_A$ in terms of $\ov\tau$ or $\ov\la$.
\blem\label{lem:etaA} Let $A$ be a bounded Borel set. Assume \eqref{eq:lambda2d} and
$\int_{(1,\infty)} (\log z)^{d/2} \lambda(\dd z) < \infty$.

 
(i) If $m_{2/d}(\lambda) < \infty$ or  $\overline \lambda (r) = \ell(r) r^{-2/d}$ and $\ell$ is slowly varying with   $\int_1^\infty \ell(u) u^{-1} \,\dd u = \infty$, then
\[
\ov\eta_A(r) \sim 
 |\ov A| \, \overline \tau(r),\qquad r\to\infty.
\]

(ii) If $\overline \lambda (r) = \ell(r) r^{-\alpha}$ 
for $\alpha \in (0,\frac 2d)$, where $\ell$ is slowly varying, then
\[
\ov \eta_A(r)\sim 
\ov\la(r)
\int_0^t\int_{\R^d}(2\pi\kappa s)^{-\al d/2}
e^{-\frac{\alpha \mathrm{dist}(y,A)^2}{2\kappa s}} 
\,\dd s\,\dd y,\qquad r\to\infty.
\] 
\elem

\begin{theorem} \label{cor:I-1-tail} Let $A$ be a bounded Borel set. Assume \eqref{eq:lambda2d} and
$\int_{(1,\infty)} (\log z)^{d/2} \lambda(\dd z) < \infty$. If
 $d=1$, 
further assume that 
\beq\label{eq:q} \exists q\in(0,2):\  M_{q}(1)<\infty.\eeq
Then under the assumptions of Lemma~\ref{lem:etaA} (i) or (ii) we have that
\[ \p \biggl( \sup_{x \in A} Y(t,x) > r \biggr) \sim 
 \overline \eta_A(r),\qquad r\to\infty.
\]
\end{theorem}

\brem As the proof shows, even without the assumptions of Lemma~\ref{lem:etaA} (i) and (ii), the statement of  Theorem~\ref{cor:I-1-tail} continues to hold provided $\ov\eta_A$ is subexponential. We were not able to prove or disprove this in general. 
\erem

\section{Growth rate} \label{sect:growth}

In what follows we assume \eqref{eq:lambda2d}. 
For $r > 0$ and $0 \leq r_1 < r_2$, we write
\begin{equation} \label{eq:defV}
  V(r)  = \biggl\{ (s,z) \colon \frac{z}{(2\pi \kappa s)^{d/2}} > r, \ s \leq t \biggr\}, \quad V(r_1, r_2) = \biggl\{ (s,z) \colon \frac{z}{(2\pi \kappa s)^{d/2}} \in (r_1, r_2],
\ s \leq t\biggr\}.
\end{equation}
Recalling \eqref{eq:deftau} we have for $r > 1$
\beq\label{eq:recall}( \mathrm{Leb} \times \lambda) ( V(r)) = 
\overline \tau (r).
\eeq
Note that $\overline \tau(r)$ is a continuous strictly decreasing 
function, with $\overline \tau (\infty) = 0$ and $\overline \tau(0+) =  \infty$
whenever $\lambda((0,1)) = \infty$.
If $m_{2/d}(\lambda) < \infty$, then by \eqref{eq:tau-form}
\begin{equation} \label{eq:Vbound}
\overline \tau(r) 
  \leq  r^{-2/d} (2\pi\kappa)^{-1}m_{2/d}(\lambda).
\end{equation}
From \eqref{eq:tau-form} we further see that whenever 
$\int_{(0,1]} z^{2/d} \,\lambda(\dd z) = \infty$
we have for any $r > 0$
\begin{equation} \label{eq:infinite}
\sup_{y: |x-y| \leq r} Y(t,y) = \infty.
\end{equation}
Therefore, our standing assumption \eqref{eq:lambda2d} is optimal 
for $d \geq 3$ and almost optimal for $d =2$.
For a more general result in this direction, see  \cite[Thm.~3.3]{CDH}. 
Furthermore, by    \cite[Thm.~3.1]{CDH},
if $\int_{(0,1]} z^p \,\lambda(\dd z) < \infty$ for some $p < \frac2d$,
then for any fixed $t$ the function $x \mapsto Y(t,x)$ is a.s.~continuous.

If
$\int_0^\infty z^{2/d}\, \lambda(\dd z) = m_{2/d}(\lambda) < \infty$,
the non-Gaussian analogue of \eqref{eq:Gauss} (see also \cite[Thm.~1.2]{Conus13}) reads as 
follows.

\begin{theorem} \label{thm:limsup}
Let $f: (0,\infty) \to (0,\infty)$ be nondecreasing and assume that \eqref{eq:lambda2d} holds.
If $d=1$, further assume \eqref{eq:q}. Then almost surely
\[
\limsup_{x \to \infty} 
\frac{\sup_{|y| \leq x} Y(t,y)}{f(x)} = \infty
\qquad \text{or} \qquad 
\limsup_{x \to \infty} 
\frac{\sup_{|y| \leq x} Y(t,y)}{f(x)} = 0,
\]
according to whether the following integral diverges or converges:
\begin{equation} \label{eq:int-cond}
\int_1^\infty r^{d-1}\ov \tau(f(r)) \,\dd r.
\end{equation}
\end{theorem}

The result says that there is no proper normalization.
If $m_{2/d}(\la)<\infty$, then almost surely there are infinitely many peaks in $B(x) = \{ y: |y| \leq x \}$ 
that are larger than $x^{d^2/2} (\log x)^{d/2}$ but only finitely many that are larger
than $x^{d^2/2} (\log x)^{d/2+\varepsilon}$.

\begin{remark}
If the L\'evy measure is small in the sense that $m_{2/d}(\lambda) < \infty$,
then the large peaks of $Y(t,x)$ are caused by points very close to the time 
$t$. (If we remove jumps close to $(t,x)$, this is equivalent to removing the 
singularity of $g$ in \eqref{eq:defY}. The local spatial supremum of the resulting process 
would have a finite moment of order $\frac 2d$. In particular, its tail probability would be 
$o(r^{-2/d})$, which by the arguments of the proof of the theorem implies that the peaks 
will be of smaller order.)
However, if $\overline \lambda \in \mathcal{RV}_{-\alpha}$ with $\alpha < \frac2d$,
then the peaks are caused by large jumps, which are not necessarily close to $t$. 
Indeed, assume that the integral in \eqref{eq:int-cond} diverges.
For some $\delta \in (0,t)$ and large $K > 0$ define
\[
\widetilde A_n = 
\left\{ \mu( \{ (s,y,z)  \colon s \leq t- \delta,\  |y| \in [n,n+1],\ z > K f(n) \} ) 
\geq 1 \right\}.
\]
Then 
as $n \to \infty$,
\[
 \p ( \widetilde A_n ) \sim C n^{d-1} \overline \lambda ( f(n)) \sim 
C n^{d-1} \overline \tau(f(n)),
\]
showing that $\sum_{n=1}^\infty \p(\widetilde A_n) = \infty$. 
By the second  Borel--Cantelli 
lemma $\widetilde A_n$ occurs infinitely often.
\end{remark}

In line with the previous remark we show in our next and final result that the 
largest peaks of $x\mapsto Y(t,x)$ are typically \emph{not} attained at integer locations 
if $m_{2/d}(\la)<\infty$. To this end, introduce the process
\beq\label{eq:Y0}Y_0(t,x)=\begin{dcases} \int_0^t \int_\R \int_{(0,\infty)} g(t-s,x-y)\bone_{\{\lvert x-y\rvert\leq \frac12,\ g(t-s,x-y)z>1\}}\,\mu(\dd s,\dd y,\dd z),&\text{if } d=1,\\  \int_0^t \int_\R \int_{(0,\infty)} g(t-s,x-y)\bone_{\{\lvert x-y\rvert\leq \frac12\}}\,\mu(\dd s,\dd y,\dd z),&\text{if } d\geq2,\end{dcases} \eeq
which is infinitely divisible with Lévy measure
$$ \eta_0(B)=\nu(\{(s,y,z): s\leq t,\ \lvert y\rvert\leq \tfrac12,\ g(s,y)z\in B\cap (\bone_{\{d=1\}},\infty)\}).$$

\bthm\label{thm:peakZd}
	Let $f: (0,\infty) \to (0,\infty)$ be nondecreasing and assume that \eqref{eq:existence-cond} holds.
	If $d=1$, further assume \eqref{eq:q}. Then 
	\begin{align*}
	\int_1^\infty r^{d-1}\ov \eta(f(r)) \,\dd r<\infty &\implies	\limsup_{x \to \infty} 
	\frac{\max_{y\in\Z^d, |y| \leq x} Y(t,y)}{f(x)} = 0,\\
	\int_1^\infty r^{d-1}\ov \eta_0(f(r)) \,\dd r=\infty&\implies \limsup_{x \to \infty} 
	\frac{\max_{y\in\Z^d, |y| \leq x} Y(t,y)}{f(x)} = \infty.
	\end{align*}
\ethm

The result above is optimal if   $\overline \eta(r) \asymp \overline \eta_0(r)$
(i.e., $0 < \liminf_{r \to \infty} \tfrac{\overline \eta_0(r)}{\overline \eta(r)}
\leq \limsup_{r \to \infty} \tfrac{\overline \eta_0(r)}{\overline \eta(r)}  < \infty$).
We end with a sufficient condition for the asymptotic equivalence of $\ov \eta$ and $\ov 
\eta_0$ and an example where they are not.

\blem\label{lem:eta0}
(i) If $m_{1 +2/d}(\lambda) < \infty$, or if
there exist $\delta > 0$ and $C>0$ such that for $r,y>1$ large enough
\begin{equation} \label{eq:lambda-cond}
\frac{\overline \lambda(ry)}{\overline \lambda (r)} \leq C y^{-\delta},
\end{equation}
then 
$\overline \eta_0(r) \asymp \overline \eta(r)$ as $r\to\infty$.

(ii) Under the assumptions of Lemma \ref{lemma:eta-tail} (iii) we have $\ov\eta_0(r)=o(\ov\eta(r))$ as $r\to\infty$.
\elem

\section{Proofs} \label{sect:proofs}

\subsection{Proofs for Section \ref{sect:intro}}

\begin{proof}[Proof of Theorem \ref{thm:existence}]
We start with the part (ii). By standard results on Poisson integrals (see e.g.~\cite[Thm. 2.7]{Kyprianou}), the integral in \eqref{eq:uncomp} exists a.s.~iff
\[
\trint  ( 1 \wedge g(s,y) z  ) \,\dd s \,\dd y \,\lambda( \dd z) < \infty, 
\]
where  
$\trint = \int_0^t \int_{\R^d} \int_{(0,\infty)}$.
For any $u > 0$
\begin{equation} \label{eq:g-level}
\begin{split}
 g(s,y) \leq u \iff &~ 
s \geq (2 \pi \kappa u^{2/d})^{-1} =: H_1(u)   
\quad \text{or}\quad \Bigl\{ s \in (0, H_1(u)) \\
&\qquad \text{ and } |y| \geq  \sqrt{ - \kappa d s \log
( 2 \pi \kappa s u^{2/d} ) }
= \sqrt{\kappa d s \log (H_1(u)/s) }=: H_2(s,u)\Bigr\}.
\end{split}
\end{equation}
Note that if $z \leq (2 \pi \kappa t)^{d/2} =: D$, then $H_1(1/z) \leq t$.
Let 
\[
\begin{split}
A_1 & = \{ (s,y,z) : z \leq D, s \leq H_1(1/z), |y| \leq H_2(s,1/z) \}, \\
A_2 & = \{ (s,y,z) : z > D, s \leq t, |y| \leq H_2(s,1/z) \}
\end{split}
\]
and 
\[
\begin{split}
B_1 & = \{ (s,y,z) : t \geq s > H_1(1/z) \}, \\
B_{21} & = \{ (s,y,z) : z \leq D, s \leq H_1(1/z), |y| > H_2(s,1/z) \}, \\
B_{22} & = \{ (s,y,z) : z > D, s \leq t, |y| > H_2(s,1/z) \}.
\end{split}
\]
Then $A_1, A_2, B_1, B_{21}, B_{22}$ form a partition of
$(0,t] \times \R^d \times (0,\infty)$. Moreover, by \eqref{eq:g-level},
$1 \leq  g(s,y) z$ iff $(s,y,z) \in A_1 \cup A_2$.

Consider the upper incomplete gamma function
$
\Gamma( s, x ) = \int_x^\infty u^{s-1} e^{-u}\, \dd u
$.
For $r \leq H_1(1/z)$, by a change of variables $v = \log (H_1(1/z) /s)$,
\begin{equation} \label{eq:H2-calc-aux}
\begin{split}
& \int_0^r H_2(s,1/z)^d\, \dd s 
 = \int_0^r \left( \kappa d s \log \frac{H_1(1/z)}{s} \right)^{d/2}\, \dd s \\
& \qquad= (\kappa d)^{d/2} H_1(1/z)^{d/2 +1} \, \int_{\log \frac{H_1(1/z)}{r}}^\infty
e^{-v(d/2 +1)} v^{d/2}\, \dd v \\
& \qquad= (\kappa d)^{d/2} \left( \frac{2}{d+2} \right)^{d/2 +1} H_1(1/z)^{d/2 +1}
\Gamma ( \tfrac d2 +1, (\tfrac d2 +1) \log (H_1(1/z) /r) ).
\end{split}
\end{equation}
Therefore, on $A_1$, after simplifying the constant,
\[
\begin{split}
\trint_{A_1}   ( 1 \wedge g(s,y) z  )\, \dd s\, \dd y\, \lambda(\dd z) & = 
\int_{(0,D]} \int_0^{H_1(1/z)} v_d H_2(s, 1/z)^d\, \dd s\, \lambda(\dd z) \\
& = 
\frac{d^{d/2}}{\pi \kappa (d+2)^{d/2 +1}}
\int_{(0,D]} z^{1 + 2/d} \,\lambda(\dd z),
\end{split}
\]
where $v_d = \pi^{d/2} / \Gamma(\frac d2 +1)$ is the volume of the unit 
ball $B(1)$. We see 
that this integral is finite iff 
$\int_{(0,1]} z^{1+2/d} \,\lambda(\dd z) < \infty$.
On $A_2$,
\[
\begin{split}
& \trint_{A_2}   ( 1 \wedge g(s,y) z  )\, \dd s\, \dd y\, \lambda(\dd z) \\
& \qquad= \frac{d^{d/2}}{\Gamma(\frac d2 +1) \pi \kappa(d+2)^{d/2 +1}}
\int_{(D, \infty)}
z^{1+2/d} 
\Gamma \biggl(\tfrac d2 +1, (\tfrac d2+1) \log \frac{z^{2/d}}{2 \pi \kappa t} \biggr) \,
\lambda(\dd z).
\end{split}
\]
Since
$\Gamma(\frac d2+1, u) \sim e^{-u} u^{d/2}$ as $u \to \infty$,   
\[
z^{1+2/d} \Gamma \biggl(\tfrac d2 +1, (\tfrac d2+1) \log \frac{z^{2/d}}{2 \pi \kappa t} \biggr) 
\sim (2 \pi \kappa t)^{d/2 +1} (1 + \tfrac 2d)^{d/2} 
\left( \log z \right)^{d/2},
\]
as $z \to \infty$,
which implies that 
\[
\trint_{A_2}  ( 1 \wedge g(s,y) z  )\, \dd s \,\dd y\, \lambda( \dd z)
< \infty \iff \int_{(1,\infty)} ( \log z)^{d/2}\, \lambda(\dd z) < \infty.
\] 

On $B_1$,
\[
\trint_{B_1} g(s,y) z \,\dd s\, \dd y\, \lambda(\dd z) = 
\int_{(0,D]} z (t - H_1(1/z))\, \lambda(\dd z),
\]
which is finite iff $\int_{(0,1]} z\, \lambda( \dd z) < \infty$.
For any $ h > 0$,
\begin{equation} \label{eq:B21-aux}
\int_{|y| > h} g(s,y) \,\dd y  =
d v_d (2 \pi \kappa s)^{-d/2} \int_{h}^{\infty}   e^{-\frac{r^2}{2 \kappa s}}
r^{d-1} \,\dd r = \frac{\Gamma  ( \frac{d}{2}, \frac{h^2}{2 \kappa s}  )}{\Gamma(\tfrac d2) }. 
\end{equation}
Furthermore, for any $a > 0$,
\begin{equation} \label{eq:B21-aux2}
\int_0^a \Gamma ( \tfrac{d}{2}, \tfrac{d}{2} \log \tfrac{a}{s}  ) \,\dd s
= a  \Gamma(\tfrac d2)  ( 1 - (1 + \tfrac 2d)^{-d/2}  ).
\end{equation}
Therefore, by \eqref{eq:B21-aux} and \eqref{eq:B21-aux2},
\[
\trint_{B_{21}} g(s,y) z\, \dd s \,\dd y \,\lambda(\dd z) = 
(2 \pi \kappa)^{-1}  ( 1 - (1 + \tfrac 2d)^{-d/2}  )
\int_{(0,D]} z^{1+2/d}\, \lambda(\dd z).
\]
Finally, on $B_{22}$, we use \eqref{eq:B21-aux} and the asymptotics 
$\Gamma(\frac d2, u) \sim e^{-u} u^{d/2- 1}$ to obtain that 
\[
\trint_{B_{22}} g(s,y) z\, \dd s\, \dd y \,\lambda(\dd z) < \infty \iff \int_{(1,\infty)} (\log z)^{d/2 -1} \,\lambda(\dd z) < \infty.
\]

By \cite[Thm.\ 2.7 (ii)]{Kyprianou} the characteristic function of the 
integral in \eqref{eq:uncomp} is
\begin{equation} \label{eq:ch-Z}
\begin{split}
\exp \left\{ 
- \trint  ( 1 - e^{\ii \theta g(s,y) z}  ) \,\dd s\, \dd y\, \lambda(\dd z) 
\right\} 
 = \exp \left\{ 
- \int_{(0,\infty)} 
 ( 1 - e^{\ii \theta u}  ) \,\eta (\dd u) \right\}.
\end{split}
\end{equation}

To prove the existence of $Y(t,x)$ 
defined as a compensated integral, we  
use the stochastic integration theory of  \cite{Rajput89}.
By writing
\begin{equation} \label{eq:defY2}
\begin{split}
Y(t,x) & = m t + \int_0^t \int_{\R^d} \int_{(0,1]} g(t-s, x-y) z \, 
(\mu - \nu)(\dd s, \dd y, \dd z) \\
& \quad + 
\int_0^t \int_{\R^d} \int_{(1,\infty)} g(t-s, x-y) z \, 
\mu (\dd s, \dd y, \dd z) \\
& =: mt +  Y_s(t,x) + Y_b(t,x)
\end{split}
\end{equation}
and 
the previously proved existence result for $Y_b(t,x)$, it is enough to 
deal with $Y_s(t,x)$, that is, we may assume that there are only small jumps.
Spelling out \cite[Thm.\ 2.7]{Rajput89} to our  setting, we 
obtain that $Y(t,x)$ exists iff
\begin{equation} \label{eq:RR-cond1}
\int_0^t\int_{\R^d}\int_{(0,1]}
  g(s,y) z\bone(g(s,y) z > 1) \,  \dd s \, \dd y\,\lambda(\dd z)< \infty,
\end{equation}
and 
\begin{equation} \label{eq:RR-cond2}
\int_0^t \int_{\R^d} \int_{(0,1]}   ( 1 \wedge  g(s,y)^2 z^2  ) \,\dd s\,
 \dd y \,\lambda(\dd z ) < \infty.
\end{equation}
To check \eqref{eq:RR-cond1}, as in \eqref{eq:B21-aux} write
\[
\int_{|y| \leq h} g(s,y)\, \dd y = \frac{1}{\Gamma(\frac d2)}
\int_{0}^{\frac{h^2}{2 \kappa s}} e^{-u} u^{d/2 -1} \,\dd u.
\]
Thus, as in \eqref{eq:B21-aux2},
\[
\int_{0}^{H_1(1/z)} \int_0^{H_2^2(s,1/z) / (2 \kappa s)}
e^{-u} u^{d/2-1}\, \dd u\, \dd s 
= H_1(1/z)  ( 1 + \tfrac 2d  )^{-d/2} \Gamma(\tfrac d2),
\]
which gives that 
\[
\int_0^t\int_{\R^d}\int_{(0,1]} 
g(s,y) z\bone(g(s,y) z > 1) \,  \dd s \, \dd y\,\lambda(\dd z)
=  ( 1 + \tfrac2d  )^{-d/2} 
\int_{(0,1]} z H_1(1/z) \,\lambda(\dd z).
\]
The latter integral exists iff
$\int_{(0,1]} z^{1+2/d} \,\lambda (\dd z) < \infty$.

For \eqref{eq:RR-cond2}, by the previous calculations, we only have to deal with the 
integral on $B_1 \cup B_{21}$. 
As 
\[
\int_{\R^d} g(s,y)^2\, \dd y= 
2^{-d} (\pi \kappa)^{-d/2} 
s^{-d/2},
\]
we obtain that 
\[
\trint_{B_1} g(s,y)^2 z^2 \,\dd s \,\dd y \,\lambda(\dd z) 
< \infty
\]
iff the second part of \eqref{eq:existence-cond} holds.
Finally, for $h > 0$,
\[
\int_{|y| > h} g(s,y)^2\, \dd y = 
\frac{\Gamma(\frac d2, \frac{h^2}{\kappa s})}
{2^{d} (\pi \kappa s)^{d/2} \Gamma(\frac d2)},
\]
and for $a > 0$,
\[
\int_0^a s^{-d/2} \Gamma(\tfrac d2, d \log \tfrac as)\, \dd s 
= a^{1-d/2} \Gamma(\tfrac d2)  
\frac{2(\frac12+\frac1d)^{-d/2}-2}{d-2},
\]
where the last fraction is $\frac12$ if $d=2$.
Thus, 
\[
\trint_{B_{21}}  g(s,y)^2  z^2 \,\dd s  \,\dd y\, \lambda(\dd z)
=  2^{-d} (\pi \kappa) ^{-d/2}\frac{2(\frac12+\frac1d)^{-d/2}-2}{d-2}
\int_{(0,1]} z^2 H_1(1/z)^{1-d/2} \, \lambda (\dd z), 
\]
which is finite iff
$\int_{(0,1]} z^{1+2/d}\, \lambda(\dd z) < \infty$. In summary, \eqref{eq:RR-cond1}
and \eqref{eq:RR-cond2} hold iff \eqref{eq:existence-cond} holds.

By \cite[Thm.\ 2.7 (iv)]{Rajput89}, the characteristic function of $Y_s(t,x)$ is
\begin{multline*}
\E [e^{\ii \theta Y_s(t,x) } ]
  =
\exp \biggl\{ -\ii \theta 
\int_0^t \int_{\R^d} \int_{(0,1]} \bone(g(s,y) z > 1)
 g(s,y) z   \, \dd s\, \dd y\,\lambda(\dd z) \\
  + \int_0^\infty \Bigl( 
e^{\ii \theta u} - 1 - \ii \theta (u \wedge 1) \Bigr) \,
\eta(\dd u) \biggr\}.
\end{multline*}
Combining with \eqref{eq:ch-Z}, we obtain \eqref{eq:charfunc}.
\end{proof}

\subsection{Proofs for Section \ref{sect:tail}}

\begin{proof}[Proof of Lemma \ref{lemma:eta-form}]
By \eqref{eq:g-level} and \eqref{eq:H2-calc-aux} and Fubini's theorem, we have
\beq\label{eq:etabar}
\begin{split}
 \overline \eta(r)
&=  \int_{(0,\infty)}  \int_0^{H_1(r/z) \wedge t} 
v_d H_2(s,r/z)^d \,\dd s \,\lambda(\dd z)
\\
& = 
\int_{(0,\infty)} 
\frac{ d^{d/2}}{\pi \kappa (d + 2)^{d/2 +1} \Gamma(\frac d2+1)} 
\left( \frac{z}{r} \right)^{1+2/d} 
\Gamma \left( \tfrac{d}{2} +1,  ( \tfrac{d}{2} +1  ) 
\log \frac{H_1(r/z)}{H_1(r/z) \wedge t} \right) 
\lambda(\dd z) \\
& = r^{-(1 + 2/d)}
\frac{ d^{d/2}}{\pi \kappa (d + 2)^{d/2 +1} \Gamma(\frac d2+1)} 
\int_0^\infty e^{-u} u^{d/2} M_{1+2/d}(Dr e^{ud/(d+2)}) \,\dd u,
\end{split}
\eeq
proving the exact formula for $\overline \eta(r)$. 
\end{proof}

\begin{proof}[Proof of Lemma \ref{lemma:eta-tail}]
(i) If $m_{1+2/d}(\lambda) < \infty$,   the asymptotic result follows 
immediately from \eqref{eq:eta-tail}.

(ii) Integration by parts gives for any $\gamma > 0$ that
\begin{equation} \label{eq:intpart}
M_{\gamma}(r) = 
\int_{(0,r]} z^{\gamma} \,\lambda(\dd z) = \int_0^r \gamma u^{\gamma -1}
\overline \lambda ( u ) \,\dd u - r^{\gamma } \overline \lambda(r).
\end{equation}
Thus, as $r \to \infty$, we have by Karamata's theorem (see \cite[Prop.\ 1.5.8 and 1.5.9a]{BGT}) that for  $\gamma > \alpha$,  
\begin{equation} \label{eq:lamb-int}
M_\gamma(r) \sim 
r^{\gamma} \overline \lambda(r)  \frac{\alpha}{\gamma - \alpha}
= r^{\gamma - \alpha} \ell(r) \frac{\alpha}{\gamma - \alpha},
\end{equation}
while for $\gamma = \alpha$,
\begin{equation} \label{eq:lamb-int2}
M_{\alpha}(r) \sim \alpha \int_1^r  {\ell(y)}y^{-1} \,\dd y = \alpha L(r),
\end{equation}
where $L$ is slowly varying and $L(r) / \ell(r) \to \infty$ as $r \to \infty$.

By \eqref{eq:lamb-int} with $\gamma = 1+\frac2d$
(or \eqref{eq:lamb-int2} for $\alpha = 1 + \frac 2d$)
and properties of slowly varying functions, 
\[
\begin{split}
M_{1+2/d}(r) \int_0^\infty e^{-u} u^{d/2} 
\frac{M_{1+2/d}(Dr e^{ud/(d+2)})}{M_{1+2/d}(r)}\, \dd u 
& \sim 
M_{1+2/d}(r) 
\int_0^\infty e^{-u} u^{d/2}  ( De^{ud/(d+2)}  )^{1+2/d -\alpha}\, \dd u \\
& = M_{1+2/d}(r) D^{1+2/d -\alpha} 
\left( \frac{d+2}{\alpha d} \right)^{d/2 +1 }  \Gamma(\tfrac d2 +1),
\end{split}
\]
where the use of Lebesgue's dominated convergence theorem is justified 
by Potter's bounds.

(iii)  For $\alpha=0$ the truncated moment $M_{1+2/d}$ is not necessarily
regularly varying, therefore more care is needed.  
First we analyze $L_0$, which is finite by \eqref{eq:existence-cond} and satisfies,
for any  large $K$,
\[
L_0(r) \geq \int_1^K \ell(ry) y^{-1} (\log y)^{d/2-1}\, \dd y \sim 
\ell(r) \int_1^K y^{-1} (\log y)^{d/2 -1} \,\dd y.
\]
Since the latter integral goes to infinity as $K \to \infty$, we obtain that 
$L_0(r) / \ell(r) \to \infty$ as $r \to \infty$. Next, for $a > 1$,
\[
L_0(ar) = \int_a^\infty \ell(ry) y^{-1} ( \log y/a)^{d/2-1} \,\dd y,
\]
thus 
\[
L_0(r) - L_0(ar) = \int_{1}^a \ell(ry) y^{-1} ( \log y)^{d/2-1} \,\dd y
+ \int_a^\infty \ell(ry) y^{-1} \Bigl( 
( \log y)^{d/2-1} - ( \log y/a)^{d/2-1} \Bigr)\, \dd y,
\]
which implies 
\[
\lim_{r \to \infty} \frac{L_0(r) - L_0(ar)}{L_0(r)} = 0,
\]
that is, $L_0(r)$ is indeed slowly varying. Furthermore, for any $a > 1$,
\begin{equation} \label{eq:L0-aux}
L_0(r) \sim \int_{a}^\infty \ell(ry) y^{-1} (\log y)^{d/2-1}\, \dd y, \qquad 
 r \to \infty.
\end{equation}

Next we turn to $\overline \eta(r)$. Changing variables 
$y = Dr e^{ud/(d+2)}$ in \eqref{eq:eta-tail}, we obtain
\[
\overline \eta (r) =
\frac{D^{1+2/d}}{d \pi \kappa \Gamma(\frac d2+1)} 
\int_{Dr}^\infty y^{-2-2/d}  \left( \log \frac{y}{Dr} \right)^{d/2}
M_{1+2/d}(y) \,\dd y.
\]
By Fubini's theorem,
\[
\begin{split}
 \int_r^\infty 
y^{-2-2/d}  \left( \log y/r \right)^{d/2} M_{1+2/d}(y)\, \dd y &= \int_{(0,r]} z^{1+2/d} \,\lambda(\dd z)\, r^{-1-2/d} \int_{1}^\infty u^{-2-2/d}
(\log u)^{d/2}\, \dd u \\
& \qquad  + 
\int_{(r,\infty)} z^{1+2/d} 
r^{-1-2/d} \int_{z/r}^\infty u^{-2-2/d} (\log u)^{d/2}\, \dd u\, \lambda(\dd z) \\
& = r^{-1-2/d} \int_{(0,\infty)} z^{1+2/d} f(1 \vee z/r) \,\lambda(\dd z),
\end{split}
\]
where $a \vee b = \max \{ a, b\}$ and
$
f(y) = \int_{y}^\infty u^{-2-2/d} (\log u)^{d/2}\, \dd u.
$
Using the fundamental theorem of calculus to write $z^{1+2/d}f(1\vee z/r)$ as an integral, exchanging the two resulting integrals by Fubini's theorem, and changing variables  $y = z/r$, we obtain
\[
\begin{split}
& r^{-1-2/d} 
\int_{(0,\infty)} z^{1+2/d} f(1 \vee z/r) \,\lambda(\dd z)  \\
&\qquad = 
\int_0^\infty \overline \lambda(ry) \Bigl[ 
(1 + \tfrac 2d) y^{2/d} f(1 \vee y) - \bone(y > 1) (\log y)^{d/2} y^{-1} \Bigr] \,\dd y \\
& \qquad= \int_0^1 \overline \lambda(ry) (1 + \tfrac 2d) y^{2/d} f(1)\, \dd y \\
& \qquad\qquad + \int_1^\infty \overline \lambda(ry) (1 + \tfrac 2d) y^{2/d} 
\int_y^\infty u^{-2-2/d} \Bigl[ (\log u)^{d/2} - (\log y)^{d/2} \Bigr]\, \dd u\, \dd y.
\end{split}
\]
Using that
\[
(1 + \tfrac 2d) y^{2/d} 
\int_y^\infty u^{-2-2/d} \Bigl[ (\log u)^{d/2} - (\log y)^{d/2} \Bigr]\, \dd u
= \tfrac{d}{2} y^{2/d} \int_{y}^\infty v^{-2-2/d} (\log v)^{d/2 -1} \,\dd v,
\]
we end up with
\begin{equation} \label{eq:eta-taila0}
\begin{split}
\overline \eta (r/D) 
= \frac{D^{1+2/d}}{d \pi \kappa \Gamma(\frac d2+1)} 
\bigg[ & \int_{0}^1 \overline \lambda(ry) (1+\tfrac 2d) y^{2/d} f(1)\, \dd y \\
& + \int_1^\infty \overline \lambda(ry) \tfrac{d}{2} y^{2/d} \int_y^\infty 
v^{-2-2/d} (\log v)^{d/2-1}\, \dd v \,\dd y \bigg].
\end{split}
\end{equation}
As $y \to \infty$,
\[
y^{2/d} \int_y^{\infty} v^{-2-2/d} (\log v)^{d/2 -1} \,\dd v
\sim (1 + \tfrac 2d)^{-1} y^{-1} (\log y)^{d/2 - 1},
\]
so for $K$ large enough,
\[
\begin{split}
 \int_K^\infty \overline \lambda(ry) y^{2/d} \int_{y}^\infty 
v^{-2-2/d} (\log v)^{d/2 -1} \,\dd v \,\dd y 
& \sim 
(1 + \tfrac 2d)^{-1} \int_K^\infty \overline \lambda(ry) y^{-1} (\log y)^{d/2 - 1} \,\dd y\\
&\sim (1 + \tfrac 2d)^{-1} L_0(r),
\end{split}
\]
where the last asymptotic follows from \eqref{eq:L0-aux}.
Since $\ell(r) /L_0(r) \to 0$ as $r \to \infty$, 
\[
\overline \eta (r/D) 
\sim \frac{D^{1+2/d}}{2 \pi \kappa \Gamma(\frac d2+1) ( 1 + \frac 2d)} 
L_0(r), 
\]
as claimed.
\end{proof}

\begin{proof}[Proof of Theorem~\ref{thm:Ytail}]
(i) Starting from the first line of \eqref{eq:etabar}, we can also write $\ov\eta(r)$ as 
\beq\label{eq:etabar-2}	\begin{split}
\ov\eta(r)&=v_d(\kappa d)^{d/2}\int_0^t s^{d/2}\int_{(r(2\pi\kappa s)^{d/2},\infty)}
\biggl(\log \frac{z^{2/d}}{2\pi\kappa sr^{2/d}}\biggr)^{d/2}\,\la(\dd z)\,\dd s\\
&=\frac{(2t)^{1+d/2}\kappa^{d/2}v_d}{d}r^{-1-2/d}\int_0^r v^{2/d} 
\int_{(D v,\infty)} \biggl(\log \frac{z}{D v}\biggr)^{d/2}\,\la(\dd z)\,\dd v,
\end{split}\eeq
where we changed variables $v=(s/t)^{d/2}r$ to go from the first to the second line.
By the fundamental theorem of calculus we have 
$\log \ov \eta(r)=C+\int_1^r \xi(v)v^{-1}\,\dd v$ with
\begin{align*} C&=\log \frac{(2t)^{1+d/2}\kappa^{d/2}v_d}{d} + \log \int_0^1 v^{2/d} 
\int_{(D v,\infty)} \biggl(\log \frac{z}{D v}\biggr)^{d/2}\,\la(\dd z)\,\dd v,\\
\xi(v)&=\frac{v^{1 + 2/d} \int_{(D v,\infty)}  (\log 
\frac{z}{D v} )^{d/2}\,\la(\dd z)}{\int_0^v u^{2/d} 
\int_{(D u,\infty)}  (\log \frac{z}{D u} )^{d/2}\,\la(\dd z) \,\dd u}-
(1+\tfrac2d). 
\end{align*}
Since $u\mapsto  \int_{(D u,\infty)}  
(\log \frac{z}{D u} )^{d/2}\,\la(\dd z)$ is decreasing in $u$, we have 
$-(1+\frac 2d)\leq \xi(v)\leq 0$. The claim now follows from \cite[Thm.~2.2.6]{BGT}. 

(ii) By (i) and \cite[Thm.\ 2.0.7]{BGT}, $\ov\eta$ has dominated variation \cite[p.\ 
54]{BGT} and   $\ov\eta(r+s)/\ov \eta(r)\to1$ as $r\to\infty$ for any $s>0$. Hence $\ov\eta$ 
is subexponential \cite[Thm.~1]{Goldie78} and \eqref{eq:barY} follows from 
\eqref{eq:tail-barY} and \cite[Thm.~3.1]{Pakes} (see also \cite[Thm.~5.1]{Pakes2}).

(iii)
By the representation theorem of regularly varying functions 
$\overline \eta \in \mathcal{RV}_{-\alpha}$
iff $\lim_{r \to \infty} \xi(r) = 1 + \frac 2d -\alpha$. 
By Karamata's theorem
(\cite[Thm.~1.6.1]{BGT}) this holds for $\alpha < 1+\frac2d$ iff
$f \in \mathcal{RV}_{-\alpha}$, where
\beq\label{eq:f}
\begin{split}
f (r) = \int_{(r,\infty)} \biggl( \log \frac{z}{r} \biggr)^{d/2} \,\lambda(\dd z) 
= \frac{d}{2} \int_{r}^\infty \frac{\overline \lambda(z)}{z} 
\biggl( \log \frac{z}{r} \biggr)^{d/2-1}\, \dd z.
\end{split}
\eeq
Consider  the kernel $k(u) = (\log u^{-1})^{d/2 -1}\bone_{(0,1)}(u)$. Define the 
\emph{Mellin convolution} of $f_1$ and $f_2$  by
\[
f_1 \stackrel{\text{M}}{*} f_2 (r) = \int_0^\infty f_1(r/t) f_2(t) t^{-1}\, \dd t,
\]
see e.g.~\cite[Sect.\ 4]{BGT}.
With this notation $f(r) = k \stackrel{\text{M}}{*} \overline \lambda (r)$.
The \emph{Mellin transform} of $k$, that is,
\[
\breve{k}(z) = \int_0^1 t^{-z-1} (\log t)^{d/2 -1}\, \dd t
= \ii^d \sqrt{z}^{-d} \Gamma(\tfrac d2),
\]
is defined and nonzero whenever $\Re z < 0$ and $\sqrt{z}$ is chosen such that 
$\mathrm{arg}(\sqrt{z}) \in (\frac14 \pi, \frac 34\pi)$. Therefore, we can apply 
 \cite[Thm.~4.8.3]{BGT}. (It is easy to check that the Tauberian condition 
is satisfied since $\overline \lambda$ is decreasing; see also 
\cite[Exercise 1.11.14]{BGT}.) Therefore, 
$ k \stackrel{\text{M}}{*} \overline \lambda \in \mathcal{RV}_{-\alpha}$ implies 
that $\overline \lambda \in \mathcal{RV}_{-\alpha}$, as claimed. The other direction was proved in Lemma~\ref{lemma:eta-tail}.

If $\alpha = 1 + \frac2d$, then $\lim_{r\to\infty} \xi(r)=0$ iff
\[
\frac{r^{1+2/d} f(r)}{\int_0^r u^{2/d} f(u) \,\dd u } \to 0,
\]
which holds iff $\int_0^r u^{2/d} f(u)\, \dd u$ is slowly varying, see 
\cite[Thm.\ 8.3.1]{BGT} or \cite[Thm.\ 1.1]{Kevei}. Using the first identity in \eqref{eq:f}, we can easily verify that the latter  holds if $m_{1+2/d}(\lambda) < \infty$.
\end{proof}

\subsection{Proofs for Section \ref{sect:spatial}}

\begin{proof}[Proof of Theorem \ref{thm:existence-X}] Without loss of generality,  assume $\kappa=\frac{1}{2\pi}$. If $d=1$, then
by \cite[Thm.\ 2.7]{Kyprianou} $X_A(t)$ exists iff
\[
\trint    \bone(s^{-d/2}z>1,\ y \in \ov A)   \,\dd s\, \dd y\, \lambda(\dd z)
= \lvert \ov A\rvert  \int_{(0,\infty)}  (z^{2}\wedge t)\, \lambda(\dd z)
<\infty,
\]
which is true for any Lévy measure. If $d\geq2$, $X_A(t)$ exists iff
\[
\trint  ( 1 \wedge s^{-d/2} z \bone(y \in \ov A)  ) \,\dd s\, \dd y\, \lambda(\dd z)
= \lvert \ov A\rvert \int_0^t \int_{(0,\infty)} 
( 1 \wedge s^{-d/2} z  ) \,\dd s\, \lambda(\dd z)
< \infty.
\]
For $z \leq t^{d/2}$,
we have 
$ \int_0^t (1 \wedge s^{-d/2} z)\, \dd s = z^{2/d} + z \int_{z^{2/d}}^t s^{-d/2} \,\dd s$,
while for $z > t^{d/2}$, we have
$
 \int_0^t (1 \wedge s^{-d/2} z)\, \dd s = t
$.
Thus,
\[
 \int_0^t \int_{(0,\infty)} 
 ( 1 \wedge s^{-d/2} z  ) \,\dd s\, \lambda(\dd z)  = 
\int_{(0,t^{d/2}]} z^{2/d} \,\lambda(\dd z) + 
\int_{(0,t^{d/2}]} z \int_{z^{2/d}}^t s^{-d/2} \,\dd s \,\lambda(\dd z) 
+ t \,\overline \lambda(t^{d/2}),
\]
which is finite iff
\eqref{eq:lambda2d} holds.
The identity \eqref{eq:chf-X} follows from
\cite[Thm.~2.7 (ii)]{Kyprianou}.
\end{proof}

\begin{proof}[Proof of Lemma \ref{lemma:tau-tail}]
(i) is an immediate consequence of \eqref{eq:tau-form}. (ii) follows from 
\eqref{eq:tau-form} combined with Karamata's theorem.
\end{proof}

\bpr[Proof of Theorem~\ref{thm:Xtail}]
Recall \eqref{eq:tail-barX}. Then, as in Theorem~\ref{thm:Ytail}, claims (i) and (ii) 
follow by writing $\log \ov \tau(r)=C+\int_1^r \xi(u)u^{-1}\,\dd u$ with
$$ 
C= - \log (\pi \kappa d)  + 
\log \int_0^D u^{2/d-1} \ov \la(u )\,\dd u, \qquad 
\xi(u) = \frac{ (Du)^{2/d}\ov 
\la(Du)}{\int_0^{Du} v^{2/d-1}\ov\la(v)\,
\dd v}-\tfrac2d, $$
where $\xi$ satisfies   $-\frac 2d<\xi(u)\leq 0$. For (iii), we have as in Theorem~\ref{thm:Ytail} that $\overline \tau \in \mathcal{RV}_{-\alpha}$
iff $\lim_{r \to \infty} \xi(r) = 2/d - \alpha$, which
for $\alpha < \frac 2d$  holds iff $\overline \lambda \in \mathcal{RV}_{-\alpha}$,
as claimed. If $\alpha = \frac2d$ then
using \cite[Thm.\ 8.3.1]{BGT}
$\lim_{r \to \infty} \xi(r) = 0$ iff 
$\int_0^r u^{2/d -1} \overline \lambda(u) \,\dd u$ is slowly varying. This holds  if 
$m_{2/d} (\lambda) < \infty$ or $\overline \lambda \in \mathcal{RV}_{-2/d}$.
\epr

\bpr[Proof of Lemma~\ref{lem:etaA}]

(i) If $m_{2/d}(\la)<\infty$, choose $\eps>0$ and observe that for $r>1$,
\beq\label{eq:help}\begin{split} \ov\eta_A(r) &\leq  \trint \bone( (2\pi\kappa s)^{-d/2} z >r)\bone_{\{y\in A^\eps\}}\,\dd s\,\dd y\,\la(\dd z)\\
 &\qquad+\trint \bone\Bigl((2\pi\kappa s)^{-d/2} e^{-\frac{\mathrm{dist}(y,A)^2}{2\kappa s}} z>r\Bigr)\bone_{\{y\notin A^\eps\}}\,\dd s\,\dd y\,\la(\dd z)\\
&\leq \lvert A^\eps\rvert \ov\tau(r)+\trint \bone\Bigl((2\pi\kappa s)^{-d/2} e^{-\frac{\mathrm{dist}(y,A)^2}{2\kappa s}} z>r\Bigr)\bone_{\{y\notin A^\eps\}}\,\dd s\,\dd y\,\la(\dd z). 
\end{split}\eeq
where $A^\eps=\{x\in \R^d: \mathrm{dist}(x,A)<\eps\}$. Since $\trint (2\pi\kappa s)^{-1} e^{-\frac{\mathrm{dist}(y,A)^2}{d\kappa s}}\bone_{\{y\notin A^\eps\}}z^{2/d}\,\dd s\,\dd y\,\la(\dd z)<\infty$, the last term in the previous display is $o(r^{-2/d})$, which together with Lemma~\ref{lemma:tau-tail} (i) shows that $\limsup_{r\to\infty}\ov\eta_A(r)/\ov\tau(r) \leq\lvert A^\eps\rvert$, which converges to $\lvert \ov A\rvert$ by letting $\eps\to0$. The opposite relation  follows from the fact that
$$ \ov \eta_A(r)\geq  \trint \bone( (2\pi\kappa s)^{-d/2} z >r)\bone_{\{y\in \ov A\}}\,\dd s\,\dd y\,\la(\dd z) =\lvert \ov A\rvert \ov\tau(r).$$

If $\ov\la(r)=r^{-2/d}\ell(r)$, one can use Potter's bounds, dominated convergence and Lemma~\ref{lemma:tau-tail} (ii) to show that the last integral in \eqref{eq:help} is
$$ \sim r^{-2/d} \ell(r)\trint (2\pi\kappa s)^{-1} e^{-\frac{\mathrm{dist}(y,A)^2}{d\kappa s}}\bone_{\{y\notin  A^\eps\}}\,\dd s\,\dd y = o(\ov\tau(r)). $$
The remaining proof is the same as in the case $m_{2/d}(\la)<\infty$.

(ii) If $\ov\la(r)=r^{-\al}\ell(r)$ for some $\al\in(0,\frac2d)$, a direct calculation shows that for $r>1$,
\begin{align*} \ov\eta_A(r) &=\trint \bone\Bigl((2\pi\kappa s)^{-d/2} e^{-\frac{\mathrm{dist}(y,A)^2}{2\kappa s}} z>r\Bigr)\,\dd s\,\dd y\,\la(\dd z) \\
&\sim r^{-\al}\ell(r)\int_0^t\int_{\R^d}(2\pi\kappa s)^{-\al d/2}
e^{-\frac{\alpha \mathrm{dist}(y,A)^2}{2\kappa s}} 
\,\dd s\,\dd y.\qedhere\end{align*}
\epr

\bpr[Proof of Theorem~\ref{cor:I-1-tail}] 
Note that for $d \geq 2$
condition \eqref{eq:lambda2d} implies summable jumps, in which
case we assume that $Y(t,x)$ has the form \eqref{eq:defY3}. 
For $d = 1$, note that $Y(t,x)=Y'_d(t,x)+Y'_s(t,x)+Y'_b(t,x)$, where
\beq\label{eq:dec}\begin{split} 
Y'_d(t,x)&=mt+\trint g(t-s,x-y)z(\bone_{\{(2\pi\kappa(t-s))^{-1/2}z\leq 1\}}-\bone_{\{z\leq 1\}})\,\dd s\,\dd y\,\la(\dd z),\\
Y'_s(t,x)&= \trint g(t-s,x-y)z\bone_{\{(2\pi\kappa(t-s))^{-1/2}z\leq 1\}} \,(\mu-\nu)(\dd s,\dd y,\dd z),\\
Y'_b(t,x)&= \trint g(t-s,x-y)z\bone_{\{(2\pi\kappa(t-s))^{-1/2}z> 1\}} \, \mu(\dd s,\dd y,\dd z). \end{split}\eeq
A straightforward computation shows that $Y'_d(t,x)<\infty$ for all Lévy measures $\la$. 
Furthermore, by \eqref{eq:q} and the proof of \cite[Thm.\ 3.8]{islands_mult}
one can show that
\beq\label{eq:smalltail} \P\biggl(\sup_{x\in A}\, \lvert Y'_s(t,x)\rvert<\infty\biggr)= 
1. \eeq
For completeness, we 
sketch the proof. We use  \cite[Thm.~1]{MR}
(with $\alpha = p = 2$) and Minkowski's integral inequality to obtain
\begin{align*}
&\E[\lvert Y_s'(t,x)-Y_s'(t,x')\rvert^2] \\
& \qquad
\leq C 
\trint \lvert g(t-s,x-y)-g(t-s,x'-y)\rvert^2  
z^2 \bone_{\{(2 \pi \kappa (t-s))^{-1/2}z<1\}}\, \nu (\dd s, \dd y, \dd z)
\end{align*} 
for all $x,x'\in\R$. 
We have on the set $(2 \pi \kappa (t-s))^{-1/2}z<1$
\begin{align*}
\lvert g(t-s,x-y)-g(t-s,x'-y)\rvert^2 z^2
&=C(t-s)^{-1} z^2\Bigl\lvert 
e^{-\frac{\lvert x-y\rvert^2}{2 (t-s)}} - 
e^{-\frac{\lvert x'-y\rvert^2}{2 (t-s)}} 
\Bigr\rvert^2\\
&\leq C\lvert g((t-s),x-y)-g((t-s),x'-y)\rvert^q z^q,
 \end{align*}
where $q$ is the exponent from \eqref{eq:q} (which satisfies $q<2$). With this 
estimate and again  \cite[Lemme~A2]{SLB98}, we conclude that
$$ 
\E[\lvert Y_s'(t,x)-Y_s'(t,x')\rvert^2]\leq 
C \lvert x-x'\rvert^{3-q}.
$$
Since $3-q >1$,  it follows from \cite[Thm.~4.3]{Khos09} that
$$ \E\biggl[\sup_{x\in A} Y_s'(t,x)^2\biggr]\leq \E[Y_s'(t,0)^2] + 
\E\biggl[\sup_{x,x'\in A}\lvert Y_s'(t,x)-Y_s'(t,x')\rvert^2\biggr]<\infty,  $$
which shows   
\eqref{eq:smalltail}. 

Next,  choose $r>0$ such that $A\subseteq B(r) = \{ x \in \R^d\colon \lvert x\rvert\leq r \}$. Then
\begin{align*} \sup_{x\in A} Y'_b(t,x) &\leq \trint (2\pi\kappa(t-s))^{-d/2}  z\bone_{\{y\in B(r)\}}\bone_{\{ d\geq2 \text{ or } (2\pi\kappa(t-s))^{-1/2}z>1\}} \,\mu(\dd s,\dd y,\dd z)\\
 &\quad+\trint (2\pi\kappa(t-s))^{-d/2} e^{-\frac{\mathrm{dist}(y,B(r))^2}{2\kappa(t-s)}} z\bone_{\{y\notin B(r),\ d\geq2 \text{ or } (2\pi\kappa(t-s))^{-1/2}z>1\}}\,\mu(\dd s,\dd y,\dd z).
\end{align*}
The first term on the right-hand side is simply $X_{B(r)}(t)$, which is finite a.s.\ by Theorem~\ref{eq:lambda2d}. The second term has the same distribution as $Y'_b(t,0)$, which is finite a.s.\ as well. Therefore, $\sup_{x\in A} Y(t,x) <\infty$ a.s.\ for all $d$.
The assertion of the theorem now follows from Lemma~\ref{lem:etaA} (which implies that $\ov\eta_A$ is subexponential under   the stated assumptions) and \cite[Thm.~3.1]{Rosinski93}.
\epr

\subsection{Proofs for Section \ref{sect:growth}}

For $0<r<r'$ let 
$B(r, r') = \{ x \in \R^d\colon r < \lvert x\rvert \leq r' \}$.

\begin{proof}[Proof of Theorem~\ref{thm:limsup}]
First assume that   \eqref{eq:int-cond} converges and let $K>0$. We start with $d\geq2$. Since   $B(n,n+1)$ can be covered with $O(n^{d-1})$ many unit cubes and $Y$ is stationary in space,  Theorem~\ref{thm:Xtail} shows that 
\beq\label{eq:bound2}\begin{split}
	\P\Biggl(\sup_{y\in B(n,n+1)} Y(t,y)>\frac{f(n)}{K}\Biggr)&\leq Cn^{d-1} \P\Biggl(\sup_{y\in [0,1]^d} Y(t,y)>\frac{f(n)}{K}\Biggr)\\
&\leq Cn^{d-1}\P\biggl(X_{[0,1]^d}(t)> \frac{f(n)}{K}\biggr)\\
&\leq 2Cn^{d-1} \ov \tau(f(n)/K) \leq   C' n^{d-1}\ov\tau(f(n)),
\end{split}\eeq
which is summable by hypothesis. 
In the last inequality we also used that $\overline \tau$ is extended regularly 
varying. So  by the first Borel--Cantelli lemma, 
\[ \sup_{y\in B(n,n+1)} Y(t,y) > \frac{f(n)}{K}\]
only happens finitely often and hence
\[ \limsup_{x\to\infty} \frac{\sup_{\lvert y\rvert\leq x} Y(t,y)}{f(x)}\leq K^{-1}\]
almost surely, proving the claim since $K$ was arbitrary. If $d=1$, recall the decomposition \eqref{eq:dec}. We can apply \eqref{eq:bound2} to $Y'_b(t,x)$, while $\lvert Y'_d(t,x)+Y'_s(t,x)\rvert$ has a smaller tail by \eqref{eq:smalltail} (in $d=1$, the tail of $Y'_b(t,x)$ is no smaller than $Cr^{-2}$ by Lemma~\ref{lemma:tau-tail} (i)). Therefore, the final bound in \eqref{eq:bound2} remains true.

For the converse statement, assume that the integral in \eqref{eq:int-cond} diverges.  If $d=1$, we consider again the decomposition \eqref{eq:dec}. As before, we let $Y'_b(t,x)=Y(t,x)$ if $d\geq2$.
For $K >0$ large consider the events
\begin{equation} \label{eq:An-def}
A_n = \{ \mu  (\{ (s,y,z): \, (s,z) \in V(Kf(n+1)),\ y \in B(n,n+1) \}) \geq 1 \},
\quad n \geq 1.
\end{equation}
By \eqref{eq:recall} and Theorem~\ref{thm:Xtail} (i),
\[
\P(A_n)\sim 
v_d  ((n+1)^d - n^d)\ov\tau(Kf(n+1)) \geq  Cn^{d-1}\ov\tau(f(n+1)).
\]
Since the integral in \eqref{eq:int-cond} diverges, we have that
$\sum_{n=1}^\infty \p (A_n ) = \infty$. Noting that the $A_n$'s are independent, the 
second Borel--Cantelli lemma implies that $A_n$ occurs infinitely often.
On $A_n$,
\[
\sup_{y \in B(n, n+1)} Y'_b(t,y) \geq    K f(n+1).
\]
Thus, almost surely,
\[
\limsup_{x \to \infty} \frac{\sup_{|y| \leq x} Y'_b(t,y)}{f(x)} \geq 
K,
\]
which proves the claim for $d\geq2$ since $K > 0$ is arbitrarily large.

If $d=1$, note that the proof above shows that $Y'_b(t,x)$ develops infinitely many peaks larger than $x^{1/2}$ on $B(x)$ (because $\ov\tau(r)$ decreases no faster than shown in Lemma~\ref{lemma:tau-tail} (i)). So if we show that  $\lvert Y'_d(t,x)\rvert+\lvert Y'_s(t,x)\rvert$ from \eqref{eq:dec} can only have finitely many peaks of that size, then the    proof in $d=1$ will be complete. For $\lvert Y'_d(t,x)\rvert$, this is trivial. For $\lvert Y'_s(t,x)\rvert$, this is a simple consequence of \eqref{eq:smalltail} and the arguments in the first part of the proof.
\end{proof}

\begin{proof}[Proof of Theorem~\ref{thm:peakZd}] The upper bound proof is essentially the same as for Theorem~\ref{thm:limsup}, except that \eqref{eq:bound2} should be replaced by
\[\begin{split}
\P\Biggl(\max_{y\in\Z^d, y\in B(n,n+1)} Y(t,y)>\frac{f(n)}{K}\Biggr)&\leq Cn^{d-1} 
\P\Biggl(Y(t,0)>\frac{f(n)}{K}\Biggr)\\
&\leq Cn^{d-1} \ov \eta(f(n)/K) \leq   C' n^{d-1}\ov\eta(f(n)).
\end{split}\]
For the lower bound proof, if $d=1$, we consider the decomposition $Y(t,x)=At + Y''_s(t,x)+Y''_{b}(t,x)$, where $A$ is the same constant as in Theorem~\ref{thm:existence} and
\begin{align*}
	Y''_s(t,x)&=\trint g(t-s,x-y)z\bone_{\{g(t-s,x-y)z\leq1\}} \,(\mu-\nu)(\dd s,\dd y,\dd z),\\
	Y''_{b}(t,x)&=\trint g(t-s,x-y)z\bone_{\{g(t-s,x-y)z>1\}} \,\mu(\dd s,\dd y,\dd z).
\end{align*}
If $d\geq2$, we   let $Y''_{b}(t,x)=Y(t,x)$.
Clearly,
$Y''_{b}(t,x)\geq Y_0(t,x)$ from \eqref{eq:Y0} and $\P(Y_0(t,x)>r)\sim \ov\eta_0(r)$ similarly to Theorem~\ref{thm:Ytail}. Because the $(Y_0(t,x))_{x\in\Z^d}$ are independent and
\begin{align*} \sum_{n=1}^\infty \sum_{y\in \Z^d\cap B(n,n+1)}\P (Y_0(t,y)> Kf(n+1) ) &\geq C \sum_{n=1}^\infty  n^{d-1}\ov\eta_0(f(n+1)K)=\infty,\end{align*}
the second Borel--Cantelli lemma shows that 
$Y''_b(t,x)/f(x)\geq Y_0(t,x)/f(x)\geq K$ for 
infinitely many $x\in\Z^d$. If $d=1$, then as in the proof of Theorem~\ref{thm:limsup} one 
can show that the peaks of $\lvert Y''_s(t,x)\rvert$ are of lower order.
\end{proof}

\begin{proof}[Proof of Lemma~\ref{lem:eta0}]
Recall $H_1$ and $H_2$ from \eqref{eq:g-level}. For $r>1$ 
	\[
	\overline \eta_0(r) = \eta_0((r,\infty)) 
	= \int_{(0,\infty)} \int_0^{H_1(r/z) \wedge t} v_d 
	( \tfrac{1}{2} \wedge H_2(s,r/z) )^d\, \dd s\, \lambda(\dd z).
	\]
	For fixed $u > 0$ the map   $s\mapsto H_2(s, u)$
	is increasing on $[0,(2\pi \kappa e u^{2/d})^{-1}]$, and decreasing on
	$[(2\pi \kappa e u^{2/d})^{-1}, H_1(u)]$, with global maximum
	$H_2((2\pi \kappa e u^{2/d})^{-1}, u) = \sqrt{d / (2 \pi e)} u^{-1/d}$.
	In particular, $H_2(s,u) \leq \frac12$ whenever $u \geq ( 2d / (\pi e))^{d/2}$.
	Therefore, as in the proof of Lemma~\ref{lemma:eta-form},
	\[
	\begin{split}
		& \int_{(0, (\pi e/(2d))^{d/2} r]} 
		\int_0^{H_1(r/z) \wedge t} v_d 
		( \tfrac{1}{2} \wedge H_2(s,r/z) )^d \,\dd s\, \lambda(\dd z) \\
		&\qquad = 
		\frac{ d^{d/2}}{\pi \kappa (d + 2)^{d/2 +1} \Gamma(\frac d2+1)} 
		\int_{(0, (2\pi e/d)^{d/2} r]} \biggl( \frac z r\biggr)^{1+2/d}
		\Gamma \biggl( \tfrac{d}{2} +1,  ( \tfrac{d}{2} +1  ) 
		\log \frac{H_1(r/z)}{H_1(r/z) \wedge t} \biggr)\, 
		\lambda(\dd z) \\
		&\qquad \geq c_1 r^{-1-2/d} M_{1+2/d} ( c_2 r).
	\end{split}
	\]
	At the same time, if $u> 0$ is small enough, then
	\[
	\int_0^{H_1(u) \wedge t} 
	( \tfrac{1}{2} \wedge H_2(s,u) )^d \,\dd s \geq \frac{t}{3}.
	\]
	Thus there exists $c_3$ such that 
	\[
	\int_{(c_3 r, \infty)} 
	\int_0^{H_1(r/z) \wedge t} v_d 
	( \tfrac{1}{2} \wedge H_2(s,r/z) )^d \,\dd s \,\lambda(\dd z) 
	\geq c_4 \overline \lambda(c_3 r).
	\]
	It follows that there are finite constants
	$c_1, c_2, C_1, C_2>0$ depending only on $d$ and $t$ such that
\begin{equation} \label{eq:eta0-bound}
c_1 r^{-1-2/d} M_{1+2/d} ( c_2 r) + 
c_1 \overline \lambda(c_2 r)
\leq
\overline \eta_0(r) \leq 	C_1 r^{-1-2/d} M_{1+2/d} ( C_2 r) + 
C_1 \overline \lambda(C_2 r).
\end{equation}
(The second inequality is an easy consequence of the first two displays in this proof.)

From \eqref{eq:eta0-bound} and Lemma \ref{lemma:eta-tail} (i) we see 
that $\overline \eta_0 (r) \asymp \overline \eta(r)$ whenever 
$m_{1+2/d} (\lambda) < \infty$. 
If \eqref{eq:lambda-cond} holds, using 
$\Gamma(\frac d2 + 1, r) \sim e^{-r} r^{d/2}$ as $r \to \infty$, we have 
\[
\begin{split}
& \int_{(rD, \infty)} 
\left( \frac{z}{r} \right)^{1+2/d} 
\Gamma \biggl( \tfrac{d}{2} +1,  ( \tfrac{d}{2} +1  ) 
		\log \frac{H_1(r/z)}{H_1(r/z) \wedge t} \biggr)\,
		\lambda(\dd z) \\
		&\qquad \leq C \int_{(rD, \infty)}  ( \log (z/r)  )^{d/2}\, \lambda (\dd 
z)=   C \int_{rD}^\infty    ( \log (z/r)  )^{d/2-1} 
		\overline \lambda(z) z^{-1} \,\dd z \\
		& \qquad\leq C \overline \lambda(r) \int_{D}^\infty ( \log y )^{d/2-1} 
y^{-1-\delta} \,\dd y 
		= C \overline \lambda(r),
	\end{split}
	\]
	which implies
	$\overline \eta_0(r) \asymp \overline \eta(r)$.

On the other hand,  (ii) follows from
 Lemma \ref{lemma:eta-tail} (iii) and \eqref{eq:eta0-bound}.
\end{proof}

\medskip
\noindent \textbf{Acknowledgments.}
PK's research was supported by the János Bolyai Research Scholarship of the 
Hungarian Academy of Sciences.

\bibliographystyle{abbrv}
\bibliography{as-heat}
 
\end{document}